\documentclass[12pt, reqno]{amsart}
\usepackage{amstext,amsthm,amsmath,amsfonts,amssymb,amscd,latexsym,stmaryrd}
\usepackage[letterpaper, margin=1.25in]{geometry}
\numberwithin{equation}{section}
\usepackage{color}
\usepackage{epsfig}
\usepackage{cite}
\usepackage[colorlinks=true, pdfstartview=FitV, linkcolor=blue, citecolor=blue, urlcolor=blue]{hyperref}
\usepackage[nameinlink]{cleveref}
\usepackage[all,cmtip]{xy}
\usepackage{geometry}
\usepackage{youngtab}

\newcommand{\C}{{\mathbb C}}
\newcommand{\R}{{\mathbb R}}
\newcommand{\Z}{{\mathbb Z}}

\newcommand{\Q}{{\mathbb Q}}

\newcommand{\F}{{\mathbb F}}

\newcommand{\Ann}{\operatorname{Ann}}
\newcommand{\flo}[1]{\left\lfloor\frac{#1}{2}\right\rfloor}

\newcommand{\Hess}{\operatorname{Hess}}
\newcommand{\sgn}{\operatorname{sgn}}

\renewcommand{\P}{{\mathbb P}}

\newtheorem{theorem}{Theorem}[section]
\newtheorem{proposition}[theorem]{Proposition}
\newtheorem{lemma}[theorem]{Lemma}

\newtheorem{introthm}{Theorem}

\newtheorem{fact}[theorem]{Fact}
\newtheorem{conjecture}[theorem]{Conjecture}
\newtheorem*{question*}{Question}
\newtheorem*{answer*}{Answer}
\newtheorem*{remarks*}{Remarks}
\newtheorem*{claim*}{Claim}
\newtheorem*{remark*}{Remark}
\newtheorem*{proposition*}{Proposition}
\newtheorem*{lemma*}{Lemma}
\newtheorem*{fact*}{Fact}

\theoremstyle{definition}

\newtheorem{remark}[theorem]{Remark}
\newtheorem{example}[theorem]{Example}

\begin{document}
	\title[Almkvist's Conjecture]{Lattice Paths, Lefschetz Properties, and Almkvist's Conjecture in Two Variables}
	\author[Abdallah-McDaniel]{Nancy Abdallah and Chris McDaniel}

\address{Department of Mathematics\\
	University of Borås\\
	Borås, Sweden}
\email{nancy.abdallah@hb.se}
	
\address{Department of Mathematics\\
		Endicott College\\
		376 Hale St Beverly, MA 01915, USA.}
	\email{cmcdanie@endicott.edu}


	\maketitle
	\begin{abstract}
	We study a certain two-parameter family of non-standard graded complete intersection algebras $A(m,n)$.  In case $n=2$, we show that if $m$ is even then $A(m,2)$ has the strong Lefschetz property and satisfies the complex Hodge-Riemann relations, while if $m$ is odd then $A(m,2)$ satisfies these properties only up to a certain degree.  This supports a strengthening of a conjecture of Almkvist on the unimodality of the Hilbert function of $A(m,n)$. 
	\end{abstract}

	\section{Introduction}
	\label{sec:Intro}
	Let $\Q$ be the field of rational numbers, and fix positive integers $m,n$.  Let $R=R^n=\Q[x_1,\ldots,x_n]$ denote the polynomial ring in $n$ variables with the standard grading.  For each $0\leq i\leq n$, let $e_i=e_i(x_1,\ldots,x_n)=\sum_{1\leq j_1<\cdots<j_i\leq n}x_{j_1}\cdots x_{j_n}\in R$ be the $i^{th}$ elmentary symmetric function, and define the $m^{th}$ power elementary symmetric functions $e_i(m)=e_i(x_1^m,\ldots,x_n^m)$.  Let $\Q[e_1,\ldots,e_n]\subset R$ be the subalgebra generated by the elementary symmetric functions.  It follows from fundamental theorem of symmetric polynomials \cite[Theorems 1.1.1, 1.1.2]{Smith} that $\Q[e_1,\ldots,e_n]$ is a polynomial algebra to which every other symmetric polynomial in $R$ belongs.  In particular, $e_i(m)\in \Q[e_1,\ldots,e_n]$ for each $0\leq i\leq n$, and let $(e_1(m),\ldots,e_n(m))\subset \Q[e_1,\ldots,e_n]$ be the ideal they generate.  The main object of study in this paper is the quotient algebra 
	\begin{equation}
		\label{eq:Amn}
		A(m,n)=\frac{\Q[e_1,\ldots,e_n]}{(e_1(m),\ldots,e_n(m))}.
	\end{equation}
	Using invariant theory, one can show that $A(m,n)$ is a graded Artinian complete intersection for all positive integers $m,n$; see \Cref{sec:Appendix} for further details.  One can show, e.g. \cite[Theorem 2.21 and Proposition 3.1]{MCIM}, that the Hilbert polynomial $H(m,n;t)$ for $A(m,n)$ is given by the formula 
	\begin{equation}
		\label{eq:Hilbert}
		H\left(m,n;t\right)=\frac{(1-t^m)\cdots(1-t^{mn})}{(1-t)\cdots(1-t^n)}=\prod_{i=1}^n\left(1+t^i+t^{2i}+\cdots+t^{(m-1)i}\right).
	\end{equation} 

	Expanding the Hilbert polynomial in \Cref{eq:Hilbert}
	$$H(m,n;t)=\sum_{k\geq 0}^dH(m,n)_k\cdot t^k, \ \ \text{where} \  H(m,n)_k=\dim_\Q(A_k),$$ 
	we refer to its sequence of coefficients $H(m,n)=(H(m,n)_k)_{k=0}^d$ as the Hilbert function of $A(m,n)$.
	
	This family of Hilbert functions $H(m,n)$ also has a combinatorial interpretation as the generating functions for a certain two parameter family of integer partitions; see \Cref{sec:Appendix} for further details.  
	In a series of papers from the 1980s, G. Almkvist \cite{Alm1,Alm2,Alm3} studied these generating functions $H(m,n)$ and formulated the following remarkable conjecture:
	\begin{conjecture}[Almkvist 1985]
		\label{conj:Alm}
		\begin{enumerate}
			\item For each $m$, the Hilbert function $H\left(m,n\right)$ is unimodal for all $n\geq 11$.  
			\item Moreover if $m$ is even, then $H\left(m,n\right)$ is unimodal for all $n$.
		\end{enumerate}
	\end{conjecture} 
\Cref{conj:Alm} was verified by Almkvist in \cite{Alm3} for $3\leq m\leq 20$, as well as for $m=100$ and $101$.  For $m=2$, \Cref{conj:Alm} was proved independently by Odlyzko-Richmond \cite{OR} using analytical methods, by Stanley \cite{Stanley} using the hard Lefschetz theorem (see \Cref{sec:Appendix} for further details), and by Hughes \cite{Hughes} using properties of root systems in Lie algebras.

Following Stanley's lead, we suspect that the unimodal condition on the Hilbert functions $H(m,n)$ always stems from certain Lefschetz properties on the algebras $A(m,n)$.  
Recall that a graded Artinian Gorenstein (e.g. complete intersection) algebra $A$ of socle degree $d$ has the strong Lefschetz property (SLP) if there exists a linear form $\ell\in A_1$ such that the multiplication maps 
\begin{equation}
	\label{eq:Lef}
	\times\ell^{d-2i}\colon A_i\rightarrow A_{d-i}, \ \ 0\leq i\leq \flo{d}
\end{equation}
are all isomorphisms.  More generally, $A$ has the Lefschetz property with respect to the Hilbert function (HLP) if the maps in \Cref{eq:Lef} have rank equal to the minimum dimension of the graded components between $i$ and $d-i$.  Note that if $A$ has SLP, then its Hilbert function $H(A)$ must be unimodal, and if $H(A)$ is unimodal then SLP and HLP are equivalent.  Related to Lefschetz properties 
are the Hodge-Riemann relations.  The algebra $A$ satisfies the Hodge-Riemann relations (HRR) if the signature of the multiplication maps in \Cref{eq:Lef} equals the alternating sum of first differences of its Hilbert function.  We introduce the complex Hodge-Riemann relations (complex HRR) to mean that the maps in \Cref{eq:Lef} have signature equal to the altenating sum of even first differences of the Hilbert function.
In general, both HRR and complex HRR imply SLP, but existence of either one of HRR or complex HRR precludes the existence of the other (in high enough degrees); see \Cref{sec:Prelim} for precise statements and further discussion.  

We formulate the following as a strengthening of Almkvist's conjecture\footnote{Items (i), (ii), and (iii) of \Cref{conj:CNAlm} were posed in the paper \cite{MCIM}; item (iv) is new.}:

	\begin{conjecture}
	\label{conj:CNAlm}
	\begin{enumerate}
		\item For each $m$, the algebra $A(m,n)$ satisfies the SLP for all $n\geq 11$.
		\item Moreover if $m$ is even, then $A(m,n)$ satisfies the SLP for all $n$.
		\item For each $m$, the algebra $A(m,n)$ satisfies HLP for all $n$.
		\item If $m$ is even, then the algebra $A(m,n)$ satisfies the complex HRR.
	\end{enumerate}
\end{conjecture}
One class of graded Artinian Gorenstein algebras that satisfies SLP and HRR is the cohomology rings (in even degrees, over $\Q$) of smooth complex projective algebraic varieties.
For example, in case $m=2$, the algebra $A(2,n)$ can be identified as the cohomology ring of a certain smooth projective algebraic variety, and as such, $A(2,n)$ must have SLP and HRR (this is essentially Stanley's proof of Almkvist's conjecture for $m=2$).  On the other hand, we will see that for $m>2$ and $n=2$, $A(m,2)$ cannot satisfy HRR and hence cannot be the cohomology ring of any smooth projective algebraic variety; see \Cref{rem:NOHRR}.
In this paper we focus on the case $n=2$.  Our main results are stated below.
\begin{introthm}
	\label{introthm:Hilb}
	The Hilbert function $H(m,2)$ satisfies the formula
	$$H(m,2)_i=\flo{i+2}-\flo{i+2-m}_\star+\flo{i+2-2m}_\star,$$ 
where $\lfloor{x}\rfloor_\star =\max (0,\lfloor{x}\rfloor)$.  If $m=2m'$ is even then
$$H(m,2)=(1,1,2,2,\cdots,m'-1,m'-1,\underbrace{m',m',\cdots,m',m'}_{m+2\text{ terms}},m'-1,m'-1,\cdots,2,2,1,1),$$
and if $m=2m'-1$, then
$$H(m,2)=(1,1,2,2,\cdots,m'-1,m'-1,\underbrace{m',m'-1,\cdots,m'-1,m'}_{m\text{ terms}},m'-1,m'-1,\cdots,2,2,1,1).$$
In particular, $H(m,2)$ is unimodal if and only if $m$ is even.
\end{introthm}
In the following, we say that an algebra $A$ satisfies SLP or HRR or complex HRR up to degree $r$ if the maps in \Cref{eq:Lef} satisfy the corresponding condition for all degrees $0\leq i\leq r$; we write SLP$_r$, HRR$_r$, complex HRR$_r$, respectively.
\begin{introthm}
	\label{introthm:HRPn2}
	\begin{enumerate}
		\item If $m$ is even then $A(m,2)$ satisfies SLP.  If $m$ is odd then $A(m,n)$ satisfies SLP$_{m-1}$.
		\item The algebra $A(m,2)$ satisfies HLP for all $m$.
		\item If $m$ is even then $A(m,2)$ satisfies the complex HRR.  If $m$ is odd then $A(m,2)$ satisfies complex HRR$_{m-1}$.
	\end{enumerate}
\end{introthm}

The first step in our proof of \Cref{introthm:HRPn2} consists in finding a presentation and a Macaulay dual generator for $A(m,2)$.
\begin{introthm}
	\label{introthm:PresMD}
	The algebra $A(m,2)$ has the following presentions:
	$$A(m,2)\cong \frac{\Q[e_1,e_2]}{(f_m(e_1,e_2),e_2^m)}= \frac{\Q[e_1,e_2]}{\Ann(F_m(E_1,E_2))}$$
	where 
	$$f_m(e_1,e_2)=\sum_{k=0}^{\flo{m}}(-1)^k\frac{m}{m-k}\binom{m-k}{k}e_1^{m-2k}e_2^k$$
	and 
	$$F_m(E_1,E_2)=\sum_{n=0}^{m-1}\frac{m}{m+2n}\binom{m+2n}{n}\frac{E_1^{m+2n-1}E_2^{m-n-1}}{(m+2n-1)!(m-n-1)!},$$
\end{introthm}
Our proof of the formula for $F_m(E_1,E_2)$ relies on an interesting formula involving Catalan numbers, pointed out to us by A. Burstein \cite{Burstein} to whom we are grateful.  The connection between the algebra $A(m,n)$ and the Catalan numbers is still somewhat surprising and mysterious to us.

We establish SLP, HLP, and complex HRR by looking at the higher Hessian matrices for the Macaulay dual generator.  It turns out that these higher Hessian matrices are weighted path matrices whose determinants count certain families of what we call subdiagonal NE lattice paths that are \emph{doubly vertex disjoint}.  We prove the following general combinatorial result:
\begin{introthm}
	\label{introthm:WPM}
	Given nonnegative integers $i,m$ satisfying $0\leq i\leq \flo{3(m-1)}$, the determinant of the matrix 
	$$\left(\frac{m}{3m-2-2p-2q}\binom{3m-2-2p-2q}{m-1-p-q}\right)_{\flo{i+2-m}_\star\leq p,q\leq \flo{i}}$$
	is equal to $(-1)^{h_i}$ times the number of doubly vertex disjoint subdiagonal NE lattice path systems from 
	\begin{align*}
		\mathcal{A}^m_i =& \left\{(p,p) \ \left| \ \flo{i+2-m}_\star\leq p\leq\flo{i}\right.\right\}\\
		\mathcal{B}^m_i=& \left\{(2m-2-q,m-1-q) \ \left| \ \flo{i+2-m}_\star\leq q\leq\flo{i}\right.\right\}.
	\end{align*}
where $h_i=\flo{i+2}-\flo{i+2-m}_\star$.  Moreover this determinant is nonzero if and only if $2h_i\leq m$.   
\end{introthm}  

This paper is organized as follows.  In \Cref{sec:Prelim}, we give precise definitions of SLP, HLP, HRR, and complex HRR, and their analogues up to degree $r$.  We also discuss matrix criteria, Macaulay duality, and higher Hessians.  In \Cref{sec:AC} we prove \Cref{introthm:Hilb} (\Cref{prop:HilbertFunction}) and \Cref{introthm:PresMD} (\Cref{prop:thmB1} and \Cref{prop:thmB2}).  In \Cref{sec:BD} we prove \Cref{introthm:HRPn2} (\Cref{prop:ThmB}) and \Cref{introthm:WPM} (\Cref{lem:mDVD} and \Cref{cor:Lefschetz}).  In \Cref{sec:Appendix} we give a brief exposition of the connections between the algebra $A(m,n)$, invariant theory, and partitions.
 
\section{Preliminaries}
\label{sec:Prelim}
\subsection{Oriented AG Algebras}
By an \emph{oriented AG algebra of socle degree $d$} we mean a pair $\left(A,\int_A\right)$ consisting of a graded Artinian Gorenstein $\Q$ algebra $A=\bigoplus_{i=0}^dA_i$, together with a fixed linear isomorphism $\int_A\colon A_d\rightarrow \Q$, called an orientation, such that for each $0\leq i\leq \flo{d}$ the intersection pairing defined by multiplication in $A$
$$\xymatrixrowsep{.5pc}\xymatrix{A_i\times A_{d-i}\ar[r] & \Q\\
	(\alpha,\beta)_i\ar@{=}[r] & \int_A\alpha\beta\\}$$
is non-degenerate.  By the \emph{Hilbert function of $A$} we mean the integer sequence $H(A)=(h_0,h_1,\ldots,h_d)$ defined by $h_i=\dim_{\Q}(A_i)$.

\subsection{Lefschetz and Hodge-Riemann Properties}
For a fixed linear form $\ell\in A_1$, define the \emph{$i^{th}$ Lefschetz map} as the multiplication map 
$$\times\ell^{d-2i}\colon A_i\rightarrow A_{d-i}$$
and the \emph{$i^{th}$ Lefschetz pairing} 
$$\xymatrixrowsep{.5pc}\xymatrix{A_i\times A_{i}\ar[r] & \Q\\
	(\alpha,\beta)_i^\ell\ar@{=}[r] & \int_A\ell^{d-2i}\alpha\beta.\\}$$

We say that the pair $\left(A,\ell\right)$ satisfies the \emph{strong Lefschetz property up to degree $r$} (SL$_r$) for some $0\leq r\leq \flo{d}$ if for all $0\leq i\leq r$, the Lefschetz map is an isomorphism, or equivalently, the $i^{th}$ Lefschetz pairing is nondegenerate for all $0\leq i\leq r$.  If $(A,\ell)$ satisfies SL$_{\flo{d}}$, then we will simply say that it satisfies the strong Lefschetz property or SLP.  As a shorthand, we sometimes say that the linear form $\ell\in A_1$ is SL$_r$ or SL(=SL$_{\flo{d}}$) for $A$.

A more general property than SLP is HLP.  We say that the pair $(A,\ell)$ satisfies the \emph{Lefschetz property with respect to its Hilbert function}\footnote{In the paper \cite{MCIM}, this property was called the ``maximum Jordan type consistent with the Hilbert function''.  It was renamed as above and further studied in the later paper \cite{IMM}.} or HLP if the rank of the $i^{th}$ Lefschetz map equals the minimum value of the Hilbert function between degrees $i$ and $d-i$, i.e.
$$\operatorname{rk}\left(\times\ell^{d-2i}\colon A_i\rightarrow A_{d-i}\right)=\min\{h_i,h_{i+1},\ldots,h_{d-i}\}.$$
\begin{remark}
	\label{rem:HLP}
	Note that if the Hilbert function $H(A)$ satisfies $h_0\leq h_1\leq h_2\leq\cdots\leq h_r$ for some $0<r\leq \flo{d}$, and if $h_r=\min\{h_r,\ldots,h_{d-r}\}$, the minimum value of the Hilbert function between degrees $r$ and $d-r$, then SL$_r$ is equivalent to HLP.  If $h_r=\max\{h_k\}$ is the maximum value of the Hilbert function, then SL$_r$ is equivalent to SLP.  Finally if the Hilbert function is unimodal, then SLP and HLP are equivalent. 
\end{remark}


Define the \emph{$i^{th}$ primitive subspace with respect to $\ell$} as 
$$P_{i,\ell}=P_{i,\ell}(A)=\ker\left(\times\ell^{d-2i+1}\colon A_i\rightarrow A_{d-i+1}\right).$$
The following well known result is called the primitive decomposition of $A$ with respect to $\ell$; for a proof see, e.g. \cite{MMSW}.
\begin{lemma}
	\label{lem:primdecomp}
	If $(A,\ell)$ satisfies SL$_{r-1}$, then for each $0\leq i\leq r$, we have $\dim(P_i)=h_i-h_{i-1}$ and we have the following orthogonal decomposition with respect to the $i^{th}$ Lefschetz form
	$$A_i=P_i\oplus \ell(A_{i-1}).$$
\end{lemma}

For fixed $0\leq r\leq \flo{d}$, we say that the pair $\left(\left(A,\int_A\right),\ell\right)$ satisfies \emph{Hodge-Riemann relations up to degree $r$} (HRR$_r$) if for all $0\leq i\leq r$, the $i^{th}$ Lefschetz form is $(-1)^i$-definite on the $i^{th}$ primitive subspace, i.e. 
$$(-1)^i\int_A\ell^{d-2i}\alpha^2>0, \ \forall \alpha\in P_{i,\ell}, \ \alpha\neq 0.$$
If $r=\flo{d}$, we say that the pair $\left(\left(A,\int_A\right),\ell\right)$ satisfies the Hodge-Riemann relations or HRR.  As a shorthand, we sometimes say that the linear form $\ell\in A_1$ is HR$_r$ or HR(=HR$_{\flo{d}}$) for $A$.

A related notion is the \emph{complex Hodge-Riemann relations up to degree $r$}, which means that for each $0\leq i\leq r$, the primitive subspace $P_{i,\ell}=0$ if $i$ is odd, and the $i^{th}$ Lefschetz form satisfies 
$$(\sqrt{-1})^i\int_A\ell^{d-2i}\alpha^2>0, \ \forall \alpha\in P_{i,\ell}, \ \alpha\neq 0.$$
We shall use the shorthand notation complex HRR$_r$, or for $r=\flo{d}$, complex HRP or complex HRR.  Note that in order to have the complex HRR$_r$, it is necessary that the primitive subspace lives only in even degrees up to degree $r$, which means in particular that $h_{2j+1}-h_{2j}=0$ for all $0\leq j\leq \flo{r}$.

\subsection{Matrix Conditions}
For each $0\leq i\leq \flo{d}$, let $\mathcal{E}_i\subset A_i$ be an ordered $\Q$-basis for the degree $i$ component of $A$, and let $\mathcal{E}_i^*\subset A_{d-i}$ be its dual basis with respect to the $i^{th}$ intersection pairing on $A$.  Let $M^\ell_i(\mathcal{E}_i)$ be the $h_i\times h_i$ matrix for the $i^{th}$ Lefschetz map with respect to the bases $\mathcal{E}_i$ and $\mathcal{E}_i^*$.  In other words, $M^\ell_i(\mathcal{E}_i)$ is the matrix for the $i^{th}$ Lefschetz pairing, i.e. if $\mathcal{E}_i=\{e_p^i \ | \ 1\leq p\leq h_i\}$ then 
$$M^\ell_i(\mathcal{E}_i)=\left(\left(e^i_p,e^i_q\right)_i^\ell\right)_{1\leq p,q\leq h_i}.$$
Note that $M^\ell_i(\mathcal{E}_i)$ is a real symmetric matrix, and we define its signature as 
$$\sgn(M^\ell_i(\mathcal{E}_i))=\#(\text{positive eigenvalues})-\#(\text{negative eigenvalues}).$$
Note that the signature of $M^{\ell}_i(\mathcal{E}_i)$ is independent of the basis $\mathcal{E}_i$.


It follows from the primitive decomposition, and the fact the restriction of the $i^{th}$ Lefschetz form to the subspace $\ell(A_{i-1})\subset A_i$ equals the $(i-1)^{st}$ Lefschetz form, that if we choose the basis $\mathcal{E}_i$ compatibly with the primitive decomposition, then the matrix $M^{\ell}_i(\mathcal{E}_i)$ has the block decomposition 
\begin{equation}
	\label{eq:matrix}
	M^{\ell}_i(\mathcal{E}_i)=\left(\begin{array}{c|c} M^\ell_i(\mathcal{E}_i)|_{P_{i,\ell}} & 0 \\ 
	\hline \\
	0 & M^{\ell}_{i-1}(\mathcal{E}_{i-1})\\ \end{array}\right).
\end{equation}
Using \Cref{eq:matrix}, we arrive at the following matrix criteria for establishing SL and HRR.\footnote{Conditions  (2) and (3) for HRR have appeared elsewhere in the literature, e.g. \cite{MMSW}.} 
\begin{lemma}
	\label{lem:HRP}
	With notation as above:
	\begin{enumerate}
		\item The pair $(A,\ell)$ satisfies SL$_r$ if and only if for every $0\leq i\leq r$, $$\det(M^\ell_i(\mathcal{E}_i))\neq 0.$$
		\item The pair $\left(\left(A,\int_A\right),\ell\right)$ satisfies HRR$_r$ (resp. complex HRR$_r$) if and only if it satisfies SL$_r$ and for every $0\leq i\leq r$ 
		$$\sgn\left(M^\ell_i(\mathcal{E}_i)\right)=\sum_{j=0}^i(-1)^i(h_j-h_{j-1}), \ \left( \text{resp.} =\sum_{j=0}^{\flo{i}}(-1)^j(h_{2j}-h_{2j-1})\right).$$
		\item In addition if $\dim(P_i)\leq 1$ for each $0\leq i\leq r$, then $\left(\left(A,\int_A\right),\ell\right)$ satisfies HRR$_r$ (resp. complex HRR$_r$)if and only if it satisfies SL$_r$ and for every $0\leq i\leq r$,
		$$\sgn\left(\det\left(M^\ell_i(\mathcal{E}_i)\right)\right)=(-1)^{\flo{i+1}} \ \left(\text{resp.} \ =\left(-1\right)^{\flo{\flo{i+2}}}\right).$$   
	\end{enumerate}
	\end{lemma}
	\begin{proof}
	Assertion (1) is clear, so we move on to assertion (2).  We will prove by induction on $r$ that if the pair $\left(\left(A,\int_A\right),\ell\right)$ satisfies HRR$_r$, respectively complex HRR$_r$, then for each $0\leq i\leq r$ we have 
	$$\sgn\left(M^\ell_i(\mathcal{E}_i)\right)=\sum_{j=0}^i(-1)^j(h_j-h_{j-1}),$$
	respectively, 
	$$\sgn\left(M^\ell_i(\mathcal{E}_i)\right)=\sum_{j=0}^{\flo{i}}(-1)^j(h_{2j}-h_{2j-1}).$$
	Assume that the pair $\left(\left(A,\int_A\right),\ell\right)$ satisfies HRR$_r$, respectively complex HRR$_r$.
	For the base case $r=0$, we have $P_{0,\ell}=A_0=\Q\cdot 1$, and hence $M^\ell_0(\mathcal{E}_0)=\int_A\ell^d>0$ is a $1\times 1$ matrix whose signature is $1=h_0$ by HRR$_0$, respectively complex HRR$_0$.  Inductively, assume the implication holds for degrees $r'\leq r-1$.  Then since $\left(\left(A,\int_A\right),\ell\right)$ satisfies HRR$_r$, respectively complex HRR$_r$, it also satisfies HRR$_{r-1}$, respectively complex HRR$_{r-1}$.  Therefore by the induction hypothesis it suffices to show that the signature condition holds in degree $r$, i.e. that 
	$$\sgn\left(M^\ell_r(\mathcal{E}_r)\right)=\sum_{j=0}^i(-1)^j(h_j-h_{j-1}),$$
	respectively
	$$\sgn\left(M^\ell_r(\mathcal{E}_r)\right)=\sum_{j=0}^{\flo{r}}(-1)^{j}(h_{2j}-h_{2j-1}).$$  
	Choosing a basis $\mathcal{E}_r$ compatible with the primitive decomposition as in \Cref{lem:primdecomp}, it follows from the decomposition in \eqref{eq:matrix} that
	$$\sgn\left(M^\ell_r(\mathcal{E}_r)\right)=\sgn\left(M^\ell_r(\mathcal{E}_r)|_{P_{r,\ell}}\right)+\sgn\left(M^\ell_{r-1}(\mathcal{E}_{r-1})\right).$$
	If $\left(\left(A,\int_A\right),\ell\right)$ has HRR$_r$, then it follows that the restriction of $M^\ell_r(\mathcal{E}_r)$ to the primitive subspace is $(-1)^r$ definite, and hence  $$\sgn\left(M^\ell_r(\mathcal{E}_r)|_{P_{r,\ell}}\right)=(-1)^r(h_r-h_{r-1})$$
	and it follows from induction that 
	$$\sgn\left(M^\ell_r(\mathcal{E}_r)\right)=\sum_{j=0}^i(-1)^j(h_j-h_{j-1}).$$
	If $\left(\left(A,\int_A\right),\ell\right)$ satisfies complex HRR$_r$, then there are two cases to consider.
	
	\textbf{Case 1:}  $r$ is odd.  In this case, the primitive subspace must be zero, i.e. $P_{r,\ell}=0$, and hence it follows by induction that 
	$$\sgn\left(M^\ell_r(\mathcal{E}_r)\right)=\sgn\left(M^\ell_{r-1}(\mathcal{E}_{r-1})\right)=\sum_{j=0}^{\flo{r-1}=\flo{r}}(-1)^j(h_{2j}-h_{2j-1}).$$
	
	\textbf{Case 2:}  $r$ is even.  In this case, the restriction of $M^\ell_r(\mathcal{E}_r)$ to the primitive subspace $P_{r,\ell}$ is definite with sign $(\sqrt{-1})^r=(-1)^{\flo{r}}$.  It then follows from our inductive hypothesis that 
	$$\sgn\left(M^\ell_r(\mathcal{E}_r)\right)=(-1)^{\flo{r}}(h_r-h_{r-1})+\sum_{j=0}^{\flo{r-1}=\flo{r}-1}(-1)^j(h_{2j}-h_{2j-1}).$$
	Conversely, assume that for some fixed linear form $\ell\in A_1$ and for all $0\leq i\leq r$, we have   
	$$\sgn\left(M^\ell_i(\mathcal{E}_i)\right)=\sum_{j=0}^i(-1)^i(h_j-h_{j-1}),$$
	respectively,
	$$\sgn\left(M^\ell_i(\mathcal{E}_i)\right)=\sum_{j=0}^{\flo{i}}(-1)^j(h_{2j}-h_{2j-1}).$$
	First we show that the pair $(A,\ell)$ satisfies SL$_r$.  Suppose not, and let $0\leq j\leq r$ be the smallest degree for which the map $\times\ell^{d-2j}\colon A_j\rightarrow A_{d-2j}$ is not an isomorphism.  Then there exists a nonzero element $\alpha\in A_j$ such that $\ell^{d-2j}\cdot \alpha=0$, and hence $0\neq \alpha\in P_{j,\ell}$ and $(\alpha,\alpha)^\ell_i=0$.  On the other hand, it follows from the minimality of $j$, and \Cref{lem:primdecomp}, that $A_j=P_j\oplus\ell(A_{j-1})$ where $\dim(P_j)=h_j-h_{j-1}$, and also that the block decomposition in \eqref{eq:matrix} holds for $i=j$.  By our assumption on the signature of $M^\ell_j(\mathcal{E}_j)$, it follows that its restriction to the primitive subspace is definite, and hence $\left(\alpha,\alpha\right)^\ell_j\neq 0$ for any nonzero $\alpha\in P_{j,\ell}$, a contradiction.  Therefore $\left(\left(A,\int_A\right),\ell\right)$ satisfies SL$_r$.  Then by \Cref{lem:primdecomp}, primitive decomposition holds, and hence also the matrix decomposition in \eqref{eq:matrix} for all $0\leq i\leq r$.  In particular, it follows that the signature of the $i^{th}$ Lefschetz form restricted to the $i^{th}$ primitive space is 
	$$\sgn\left(M^\ell_i(\mathcal{E}_i)\right)=\sgn\left(M^\ell_i(\mathcal{E}_i)\right)-\sgn\left(M^\ell_{i-1}(\mathcal{E}_{i-1})\right)$$
	and the assertion (2) follows.
	
	For assertion (3), note that the restricted matrix $M^\ell_i(\mathcal{E}_i)|_{P_{i,\ell}}$ is either empty (if $P_{i,\ell}=0$) or a $1\times 1$ matrix.  It follows that the signature of the $i^{th}$ Lefschetz form on the $i^{th}$ primitive subspace is equal to the sign of the quotient of consecutive determinants
	$$\sgn\left(M^\ell_i(\mathcal{E}_i)|_{P_{i,\ell}}\right)=\sgn\left(\frac{\det\left(M^\ell_i(\mathcal{E}_i)\right)}{\det\left(M^\ell_{i-1}(\mathcal{E}_{i-1})\right)}\right)=\frac{\sgn\left(\det\left(M^\ell_i(\mathcal{E}_i)\right)\right)}{\sgn\left(\det\left(M^\ell_{i-1}(\mathcal{E}_{i-1})\right)\right)}.$$
	A straightforward inductive argument then shows that the pair $\left(\left(A,\int_A\right),\ell\right)$ satisfies HRR$_r$, respectively complex HRR$_r$, if and only if
	$$\sgn\left(M^\ell_i(\mathcal{E}_i)\right)=(-1)^{\flo{i+1}},$$ 
	respectively 
	$$\sgn\left(M^\ell_i(\mathcal{E}_i)\right)=(-1)^{\flo{\flo{i+2}}}$$
	which is assertion (3). 
\end{proof}

\begin{remark}
	\label{rem:preclude}
	Note that if $\left(\left(A,\int_A\right),\ell\right)$ has a nonzero primitive space $P_{r,\ell}$ with $r\geq 2$, then HRR$_r$ and complex HRR$_r$ are mutually exclusive.  Indeed in order to satisfy HRR$_2$ and complex HRR$_2$ condition (2) implies that the Hilbert function must satisfy $h_1=2h_2+h_0>h_2$.  On the other hand, SL$_1$ implies the multiplication map $\time\ell\colon A_1\rightarrow A_2$ is injective and hence $h_1<h_2$, a contradiction. 
\end{remark}

\subsection{Macaulay Duality and Higher Hessians}
Let $R=\Q[x_1,\ldots,x_n]$ and $Q=\Q[X_1,\ldots,X_n]$ be two graded polynomial rings in $n$-variables, regarding $Q$ as a module over $R$ by the partial differentiation action, i.e. for all $1\leq i\leq n$ and $F\in Q$,
$$x_i\circ F=\frac{\partial F}{\partial X_i},$$
and the action is extended algebraically to all of $R$.
In other words, $R$ is the polynomial ring generated by partial differential operators $\partial/\partial X_1,\ldots,\partial/\partial X_n$.  Then for each homogeneous polynomial $F\in Q_d$ of degree $d$, the set 
$$\Ann(F)=\left\{f\in R \ | \ f\circ F\equiv 0\right\}$$
is an ideal in $R$, and the quotient 
$$A_F=\frac{R}{\Ann(F)}$$
is an oriented AG algebra of socle degree $d$ with a canonical orientation given by
$$\int_A\colon A_d\rightarrow \Q, \ \ \int_A\alpha = \alpha\circ F.$$ 
In fact, every oriented AG algebra arises this way, and the polynomial $F$ is called the \emph{Macaulay dual generator for the oriented AG algebra $A_F$}.  

Given the oriented AG algebra $A=A_F$ of socle degree $d$, given $0\leq i\leq \flo{d}$, and given a basis $\mathcal{E}_i=\{b^i_p \ | \ 1\leq p\leq h_i\}\subset A_i$, define the \emph{$i^{th}$ Hessian matrix of $F$ with respect to $\mathcal{E}_i$} by $$\Hess_i(F,\mathcal{E}_i)=\left(b_p^ib^i_q\circ F\right)_{1\leq p,q\leq h_i};$$
it is an $h_i\times h_i$ matrix with entries that are homogeneous polynomials of degree $d-2i$ in $Q$.  Given a vector $\mathbf{c}=(c_1,\ldots,c_n)\in\Q^n$, we can evaluate these polynomial entries at $\mathbf{c}$ and get the evaluated Hessian
$$\Hess_i(F,\mathcal{E}_i)|_{\mathbf{c}}$$
which is a real symmetric matrix.  Assuming that the ring $A_F$ has at least one non-zero linear form, and relabeling the generators if necessary, we may assume that for some $1\leq m\leq n$ the generators $x_1,\ldots,x_m$ all have degree one.  Then for any linear form $\ell=c_1x_1+\cdots+c_mx_m$, and for any homogeneous polynomial $G\in Q_{k}$ of degree $k$, we have 
$$\ell^k\circ G=k!G(c_1,\ldots,c_m,0,\ldots,0).$$
We shall refer to the case where all variables $x_1,\ldots,x_n$ all have degree one as the \emph{standard graded case}.
Using this, one can show the following result, which appears to have been observed first by J. Watanabe \cite{W}.
\begin{lemma}
	\label{lem:Hessian}
	Given an oriented AG algebra $A=A_F$, and any linear form $\ell=\ell(\mathbf{c})=c_1x_1+\cdots+c_mx_m+0x_{m+1}+\cdots+0x_n\in A_1$, define the vector $\mathbf{c}=(c_1,\ldots,c_m,0,\ldots,0)\in\Q^n$.  Then for any $0\leq i\leq \flo{d}$ and for any choice of basis $\mathcal{E}_i\subset A_i$, we have 
	$$\Hess_i(F,\mathcal{E}_i)|_{\mathbf{c}}=(d-2i)!M^{\ell(\mathbf{c})}_i(\mathcal{E}_i).$$ 
\end{lemma}
In particular, \Cref{lem:Hessian} and \Cref{lem:HRP} give a Hessian criterion on the Macaulay dual generator of an oriented AG algebra for the SL and HRR properties.

\section{Proof of \Cref{introthm:Hilb} and \Cref{introthm:PresMD}}
\label{sec:AC}
The proof of \Cref{introthm:PresMD} comes in two parts.
\begin{proposition}
	\label{prop:thmB1}
	The algebra $A(m,2)$ has the following presentation:
	$$A(m,2)=\frac{\Q[e_1,e_2]}{(f_m(e_1,e_2),e_2^m)}$$
	where $\deg(e_1)=1$ and $\deg(e_2)=2$ and 
	$$f_m(e_1,e_2)=\sum_{k=0}^{\flo{m}}(-1)^k\frac{m}{m-k}\binom{m-k}{k}e_1^{m-2k}e_2^k.$$
\end{proposition}
\begin{proof}
	Define the symbol
	$$c_{m,k}=\begin{cases} \frac{m}{m-k}\binom{m-k}{k} & \text{if} \ 0\leq k\leq \flo{m}\\ 
		0 & \text{if} \ k<0 \ \text{or} \ \flo{m}<k\\ \end{cases}$$
	We want to show that for each positive integer $m$
	$$f_m(x+y,xy)=\sum_{k=0}^{\flo{m}}(-1)^kc_{m,k}(x+y)^{m-2k}(xy)^k=x^m+y^m.$$
	Note that if $d=\flo{m}$, then $d+1=\flo{m+2}$ for all positive integers $m$.  Also note that for $m=1$, $c_{1,0}=1$ and $f_1(e_1,e_2)=e_1$, and for $m=2$, we have $c_{2,0}=1$ and $c_{2,1}=2$, and $f_2(e_1,e_2)=e_1^2-2e_2$.  Also for every positive integer $m$ we have the recursive relation 
	$$x^{m+2}+y^{m+2}=(x^2+y^2)(x^{m}+y^m)-(xy)^2(x^{m-2}+y^{m-2}).$$
	Hence by induction on $m$, it suffices to show that our functions satisfy the following recursive formula:
	\begin{equation}
		\label{eq:recursive}
		f_{m+2}(e_1,e_2)=(e_1^2-2e_2)f_{m}(e_1,e_2)-e_2^2f_{m-2}(e_1,e_2).
	\end{equation}

	With $m$ fixed, set $d=\flo{m}$ and the left hand side of \Cref{eq:recursive} is 
	$$f_{m+2}(e_1,e_2)=\sum_{k=0}^{d+1}(-1)^kc_{m+2,k}e_1^{m+2-2k}e_2^k$$
	and the right   hand side of \Cref{eq:recursive} is 
	\begin{align*}
		(e_1^2-2e_2)\sum_{k=0}^{d}(-1)^kc_{m,k}e_1^{m-2k}e_2^k-e_2^2\sum_{k=0}^{d-1}(-1)^kc_{m-2,k}e_1^{m-2-2k}e_2^k \\ 
		=\sum_{k=0}^d(-1)^kc_{m,k}e_1^{m+2-2k}e_2^k-\sum_{k=0}^d(-1)^k2c_{m,k}e_1^{m-2k}e_2^{k+1}-\sum_{k=0}^{d-1}(-1)^kc_{m-2,k}e_1^{m-2-2k}e_2^{k+2}\\
		=\sum_{k=0}^d(-1)^kc_{m,k}e_1^{m+2-2k}e_2^k-\sum_{k=1}^{d+1}(-1)^{k-1}2c_{m,k-1}e_1^{m+2-2k}e_2^{k}-\sum_{k=2}^{d+1}(-1)^{k-2}c_{m-2,k-2}e_1^{m+2-2k}e_2^{k}\\
		=\sum_{k=0}^{d+1}(-1)^k\left(c_{m,k}+2c_{m,k-1}-c_{m-2,k-2}\right)e_1^{m+2-2k}e_2^k.
	\end{align*}
	Hence the result follows if we can prove that the symbols $c_{m,k}$ satisfy the recursive formula
	\begin{equation}
		\label{eq:recursion}
		c_{m+2,k}=c_{m,k}+2c_{m,k-1}-c_{m-2,k-2}.
	\end{equation}
	On the right hand side of \Cref{eq:recursion} we have
	\begin{align*}
		& \frac{m(m-k-1)!}{k!(m-2k)!}+2\frac{m(m-k)!}{(k-1)!(m+2-2k)!}-\frac{(m-2)(m-k-1)!}{(k-2)!(m+2-2k)!}\\
		=&  \frac{m(m-k-1)!(m+1-2k)(m+2-2k)+2km(m-k)!+k(k-1)(m-2)(m-k-1)!}{k!(m+2-2k)!}\\
		= &\frac{(m-k-1)!\left(m(m+1-2k)(m+2-2k)+2km(m-k)-k(k-1)(m-2)\right)}{k!(m+2-2k)!}\\
		= & \frac{(m-k-1)!\left((m+2)(m-k+1)(m-k)\right)}{k!(m+2-2k)!}
	\end{align*}
	which is equal to the left hand side, as desired.
\end{proof}

At this point we can prove \Cref{introthm:Hilb}.
\begin{proposition}
	\label{prop:HilbertFunction}
	For any positive integer $m\geq 1$, the algebra $A(m,2)$ has socle degree $d=3(m-1)$, and for each $0\leq i\leq d$ its Hilbert function $H(m,2)$ satisfies
	$$H(m,2)_i=\flo{i+2}-\flo{i+2-m}_\star-\flo{i+2-2m}_\star$$
	where $\flo{x}_\star=\max\{\flo{x},0\}$.
\end{proposition}
\begin{proof}
	It follows from the presentation in \Cref{prop:thmB1} that $A(m,2)$ has a monomial basis $\mathcal{E}=\{e_1^pe_2^q \ | \ 0\leq p,q\leq m-1\}$, and hence the socle generator is $e_1^{m-1}e_2^{m-1}$ in degree $m-1+2(m-1)=3(m-1)=d$.  For general $i$, $0\leq i\leq d$, $H_i^m=H(m,2)_i$ is counting the lattice points $(p,q)$ lying in the $(m-1)\times (m-1)$ square on the diagonal line $x+2y=i$.  For $0\leq i\leq m-1$, we can count these lattice points using Pick's theorem:  The area ($A$) of a lattice polygon is the number of interior lattice points ($c$) plus half the number of boundary lattice points ($b$) minus $1$, i.e. $A=c+\frac{b}{2}-1$.  In our case, for fixed $0\leq i\leq m-1$, take the smallest lattice polygon containing the four points $(i,0)$, $(i-1,0)$, $\left(0,\flo{i}\right)$, and $\left(0,\flo{i-1}\right)$, which has area $i/2$ with no interior points, and one more boundary point than the sum $H^m_i+H^m_{i-1}$.  Plugging into Pick's formula, we get the relation
	$$H_i^m+H^m_{i-1}=i+1$$
	from which it follows inductively that $H_i^m=\flo{i+2}$ for all $0\leq i\leq m-1$.  Next we assume that $i$ satisfies $m\leq i\leq 2m-1$.  Then as above, we compute using Pick's formula, the number of lattice points on the diagonal line segment from $(i,0)$ to $\left(0,\flo{i}\right)$ is $\flo{i+2}$, but now we subtract off those lattice points that do not lie in the $(m-1)\times(m-1)$ square which is computed by shifting over by $m$, hence $H_i^m=\flo{i+2}-\flo{i+2-m}$ for $m\leq i\leq 2m-1$.  Finally we repeat this procedure for $2m\leq i\leq 3m-3$:  the number of lattice points on the diagonal line $x+2y$ between $(i,0)$ and $\left(0,\flo{i}\right)$ in the $(m-1)\times(m-1)$ square is $\flo{i+2}-\flo{i+2-m}-\flo{i+2-2m}$.      
\end{proof}
Here is the second part of \Cref{introthm:PresMD}.  As in \Cref{sec:Prelim} consider the two non-standard graded polynomial rings $R=\Q[e_1,e_2]$ and $Q=\Q[E_1,E_2]$ where $\deg(e_i)=i=\deg(E_i)$ for $i=1,2$ and the lower case letters act on the upper case letters by partial differentiation, i.e. 
$e_i\circ F(E_1,E_2)=\partial F/\partial E_i$ for $i=1,2$.
\begin{proposition}
	\label{prop:thmB2}
	The algebra $A(m,2)$ has the following presentation:
	$$A(m,2)=\frac{\Q[e_1,e_2]}{\Ann(F_m(E_1,E_2))}$$
	where 	$$F_m(E_1,E_2)=\sum_{n=0}^{m-1}\frac{m}{m+2n}\binom{m+2n}{n}\frac{E_1^{m+2n-1}E_2^{m-n-1}}{(m+2n-1)!(m-n-1)!}.$$ 
\end{proposition}
\begin{proof}
	Let $F=F_m(E_1,E_2)$ be as above, and note that $e_2^m\circ F\equiv 0$.  Hence it suffices to show that $f_m(e_1,e_2)\circ F\equiv 0$.  Consider the monomial basis of the degree $2m-3$ component $\Q[e_1,e_2]$\footnote{indexed here slightly differently than we have done it in the next section}:
	$$\mathcal{E}_{2m-3}=\left\{e_1^{2i-1}e_2^{m-i-1} \ | \ 1\leq i\leq \flo{m}\right\}.$$
	Then it suffices to show that for every $1\leq i\leq \flo{m}$, 
	$$e_1^{2i-1}e_2^{m-i-1}f_m(e_1,e_2)\circ F_m(E_1,E_2)= 0.$$  
	We have 
	\begin{equation}
		\label{eq:abf}
		e_1^{2i-1}e_2^{m-i-1}f_m(e_1,e_2)=\sum_{k=0}^{\flo{m}}(-1)^k\frac{m}{m-k}\binom{m-k}{k}e_1^{m-2k+2i-1}e_2^{k+m-i-1}
	\end{equation} 
	and applying \Cref{eq:abf} to $F_m(E_1,E_2)$, the problem comes down to proving the following binomial identities:   
	\begin{equation}
		\label{eq:zero}
	\sum_{k=0}^{i}(-1)^k\frac{m}{m-k}\binom{m-k}{k}\frac{m}{m+2(i-k)}\binom{m+2(i-k)}{i-k}=0, \ \forall 1\leq i\leq \flo{m}.
	\end{equation}
	These authors posted this problem to MathStackExchange \cite{MSE}, and solutions were provided by M. Riedel \cite{Riedel} and A. Burstein \cite{Burstein} to whom these authors are grateful.  We have adopted the proof given by A. Burstein, using some general results on Lagrange inversion and Catalan generating functions from a paper of I. Gessel \cite{Gessel}.  
	
	Let $C_n=\frac{1}{2n+1}\binom{2n}{n}$ denote the $n^{th}$ Catalan number, and let $C(x)=\sum_{n\geq 0}C_nx^n$ be its generating function.  The following formula for the $m^{th}$ power of $C(x)$ appears in Gessel's paper as \cite[Equation 2.3.2]{Gessel}.
	\begin{equation}
		\label{eq:Cm}
		C(x)^m=\sum_{n\geq 0}\frac{m}{m+2n}\binom{m+2n}{n}x^n.
	\end{equation}
	Using a form of Lagrange inversion, also in \cite{Gessel}, we compute the following formula for its reciprocal:
	\begin{equation}
		\label{eq:recCm}
		\frac{1}{C(x)^m}=\sum_{n= 0}^{\flo{m}}(-1)^n\frac{m}{m-n}\binom{m-n}{n}x^n+\sum_{n\geq\flo{m}+1} b_nx^n
	\end{equation}
	for some coefficients $b_n\in\R$.  Then, using the coefficient extractor function, where $[x^n](f)$ denotes the coefficient of $x^n$ in the power series expansion of $f$, we have 
	\begin{align*}
		[x^i]\left(C(x)^m\cdot\frac{1}{C(x)^m}\right)= 0=&  \sum_{k=0}^i(-1)^k\frac{m}{m-k}\binom{m-k}{k}\frac{m}{m+2(i-k)}\binom{m+2(i-k)}{i-k}
	\end{align*}
		for all $1\leq i\leq \flo{m}$, and this completes the proof.
	\end{proof}
	

\section{Proofs of \Cref{introthm:HRPn2} and \Cref{introthm:WPM}}
\label{sec:BD}
Now that we have the Macaulay dual generator of $A(m,2)$ we can compute its Hessian matrices.  By \Cref{prop:HilbertFunction}, we need only consider Hessians in degree $0\leq i\leq \flo{3(m-1)}<2m-1$.  For each $0\leq i\leq 2m-1$, let us fix the ordered monomial basis in degree $i$: 
$$\mathcal{E}_i=\left\{b^i_p=e_1^{i-2p}e_2^p \ | \ \flo{i+2-m}_\star\leq p\leq \flo{i}\right\}.$$
Note this is a basis by \Cref{prop:HilbertFunction} and \Cref{prop:thmB1}.
\begin{lemma}
	\label{lem:Hessian1}
	For each $i$, $0\leq i\leq \flo{3(m-1)}$, the $i^{th}$ Hessian matrix for $F_m(A,B)$ with respect to the monomial basis $\mathcal{E}_i$, evaluated at the point $\mathbf{c}=(1,0)$ is 
$$\operatorname{Hess}_{i}(F_m,\mathcal{E}_i)|_{\mathbf{c}}=\frac{1}{(3m-3-2i)!}\left(\frac{m}{3m-2-2p-2q}\binom{3m-2-2p-2q}{m-1-p-q}\right)_{\flo{i+2-m}_\star\leq p,q\leq \flo{i}}$$	
\end{lemma}
\begin{proof}
	By the formula for the $i^{th}$ Hessian, we have 
	$$\Hess_i(F_m,\mathcal{E}_i)=\left(e_1^{2i-2p-2q}e_2^{p+q}\circ F_m\right)_{\flo{i+2-m}_\star\leq p,q\leq \flo{i}}.$$
	Evaluating at $\mathbf{c}=(1,0)$, we have   
	\begin{align*}
	 e_1^{2i-2p-2q}e_2^{p+q}\circ F_m|_\mathbf{c}=\\
		e_1^{2i-2p-2q}e_2^{p+q}\circ \left.\left(\sum_{n=0}^{m-1}\frac{m}{m+2n}\binom{m+2n}{n}\frac{E_1^{m+2n-1}E_2^{m-n-1}}{(m+2n-1)!(m-n-1)!}\right)\right|_\mathbf{c}\\
		= e_1^{2i-2p-2q}e_2^{p+q}\circ\left.\left( \frac{m}{3m-2p-2q-2}\binom{3m-2p-2q-2}{m-p-q-1}\frac{E_1^{3m-2p-2q-3}E_2^{p+q}}{(3m-2p-2q-3)!(p+q)!}\right)\right|_\mathbf{c}\\
		=\frac{m}{3m-2p-2q-2}\binom{3m-2p-2q-2}{m-p-q-1}\frac{1}{(3m-3-2i)!}	
	\end{align*}
and the result follows.
\end{proof}

Recall that given a directed acyclic graph $G=(V(G),E(G))$ with vertex set $V(G)$ and directed edge set $E(G)$ together with a weight function $\omega\colon E(G)\rightarrow R$ taking values in some fixed commutative ring $R$ (for our purposes, $R=\Q$ will suffice), and further given row and column vertex subsets (finite and of the same cardinality will suffice for our purposes) $\mathcal{A}=\{A_1,\ldots,A_r\},\mathcal{B}=\{B_1,\ldots,B_r\}\subset V(G)$, one can form the associated (square $r\times r$) \emph{weighted path matrix}
$$W(G,\omega,\mathcal{A},\mathcal{B})=\left(\sum_{P\colon A_p\rightarrow B_j}\omega(P)\right)_{1\leq p,q\leq r}$$
where the sum is over all directed paths $P\colon A_p\to B_q$ from the row vertex $A_p$ to the column vertex $B_q$ in $G$ and $\omega(P)$ is the path-weight defined to be the product of the weights of all directed edges in the path, $\omega(P)=\prod_{e\in P}\omega(e)$.  

A \emph{path system from $\mathcal{A}$ to $\mathcal{B}$}, denoted by $\mathcal{P}\colon \mathcal{A}\rightarrow \mathcal{B}$, is a collection of paths $\mathcal{P}=\{P_i\colon A_i\rightarrow B_{\sigma(i)} \ | \ 1\leq i\leq r\}$ from the row vertices $\mathcal{A}$ to the column vertices $\mathcal{B}$ forming a bijection between the two sets.  Since $\mathcal{A}$ and $\mathcal{B}$ are indexed with the same set $\{1,\ldots,r\}$, we can associate a permutation $\sigma\in\mathfrak{S}_r$ to every path system, and we define the \emph{sign of the path system} to be the sign of the permutation, i.e. $\sgn(\mathcal{P})=\sgn(\sigma)$.  Define the \emph{weight of the path system} to be the product of the weights of its paths, denoted $\omega(\mathcal{P})=\prod_{i=1}^n\omega(P_i)$.  We say that a path system is \emph{vertex disjoint} if no two paths in the path system share a common vertex.  The following very useful result is due to Lindstr\"om \cite{Lindstrom} and Gessel-Viennot \cite{GV}.
\begin{fact}
	\label{fact:LGV}
	The determinant of the weighted path matrix is 
	$$\det\left(W(G,\omega,\mathcal{A},\mathcal{B})\right)=\sum_{\substack{\mathcal{P}\colon\mathcal{A}\rightarrow\mathcal{B}\\ \text{vertex disjoint}\\}}\sgn(\mathcal{P})\omega(\mathcal{P})$$
	where the sum is over all vertex disjoint path systems from $\mathcal{A}$ to $\mathcal{B}$.
\end{fact}  
Typically \Cref{fact:LGV} is useful in cases where the set of vertex disjoint path systems all have the same sign, although this will not quite the case in our situation.

In our situation, we shall consider the following directed acyclic graph $G$.  Let $V(G)=\Z^2$, the set of lattice points in the plane, and let $E(G)$ denote the set of North edges $(i,j)\rightarrow (i,j+1)$ together with the set of East edges $(i,j)\rightarrow (i+1,j)$; in this case the directed paths are NE lattice paths.  A useful and thorough reference for NE lattice path enumeration is C. Krattenhaler's comprehensive paper \cite{K}.  

We will actually restrict our attention to the induced sugraph $G'$ with vertex set $V(G')=\left\{(i,j) \ | \ i\geq j\right\}\subset\Z^2$ consisting of lattice points lying on or below the main diagonal line $y=x$.  Hence in the directed graph $G'$, directed paths are \emph{subdiagonal} NE lattice paths.  Define the constant edge-weight function $\omega(e)\equiv 1$, so that every directed edge, and hence also every directed path, has weight equal to one. Our row and column vertex sets depend on the nonnegative integers $i,m$ satisfying $0\leq i\leq \flo{3(m-1)}$, and are given by
\begin{align*}
	\mathcal{A}=\mathcal{A}_i^m =& \left\{(p,p) \ \left| \ \flo{i+2-m}_\star\leq p\leq\flo{i}\right.\right\}\\
	\mathcal{B}= \mathcal{B}_i^m=& \left\{(2m-2-q,m-1-q) \ \left| \ \flo{i+2-m}_\star\leq q\leq\flo{i}\right.\right\}.
\end{align*}   
Note the vertices $\mathcal{A}^m_i$ lie on the diagonal line $y=x$, and the vertices $\mathcal{B}^m_i$ lie on the shifted diagonal line $y=x-(m-1)$.  The following is \cite[Corollary 10.3.2]{K}.
\begin{fact}
	\label{fact:K}
	The number of subdiagonal NE lattice paths from a vertex $A=(a,a)\in V(G')$ on the diagonal line $y=x$ to an arbitrary vertex $B=(b,c)\in V(G')$ is equal to 
	$$\frac{b-c+1}{b+c-2a+1}\binom{b+c-2a+1}{c-a}.$$ 
\end{fact}

Equipped with \Cref{fact:K}, we can identify the Hessian with a multiple of the weighted path matrix described above $W(G',\omega'\equiv 1,\mathcal{A}_i^m,\mathcal{B}_i^m)$.  This is stated formally below.
\begin{lemma}
	\label{lem:PHessWPM}
	For each $0\leq i\leq \flo{3(m-1)}$, the $i^{th}$ Hessian of $F_m(E_1,E_2)$ with respect to the ordered monomial basis $\mathcal{B}_i$ satisfies 
	$$\Hess_{i}(F_m,\mathcal{E}_i)|_{(1,0)}=\frac{1}{(3m-3-2i)!}W(G',\omega'\equiv 1, \mathcal{A}^m_i,\mathcal{B}^m_i).$$ 
\end{lemma} 
\begin{proof}
	The $(p,q)$-entry of $W$ is $\sum_{P\colon A_p\rightarrow B_q}\omega'(P)$ which, in this case, equals the number of subdiagonal NE lattice paths from $A_p$ to $B_q$, which, according to \Cref{fact:K}, equals
	$$\frac{m}{3m-2-2q-2p}\binom{3m-2-2q-2p}{m-1-q-p}.$$
	The result now follows immediately from \Cref{lem:Hessian1}. 
\end{proof}

Note that our setup allows for vertex disjoint path systems from $\mathcal{A}^m_i$ to $\mathcal{B}_i^m$ with distinct signs, and hence we expect further cancellations to occur in the sum in \Cref{fact:LGV}.  It turns out however that we can keep track of the cancellations that occur, and characterize the terms that persist.  To this end, let us define the \emph{doubly vertex disjoint path systems}, as follows.

Generally, given a NE lattice path in $G'$, say $P\colon A=v_0\rightarrow\cdots\rightarrow v_{N}=B$ from a vertex $A=(a,a)$ on the diagonal line $y=x$ to $B=(b+m-1,b)$ on the shifted diagonal line $y=x-m+1$, define its \emph{initial segment} $I\colon A=v_0\rightarrow v_j$ where $v_j$ is the first vertex in $P$ that lies on the shifted diagonal $y=x-m+1$, and define a \emph{primitive segment} of $P$ to be a segment $S\colon v_k\rightarrow v_{k+1}\rightarrow\cdots v_{\ell-1}\rightarrow v_\ell$ such that the two endpoints of $S$, $v_k$ and $v_\ell$, are the only vertices of $S$ to lie on the shifted diagonal.  We shall refer to an \emph{upper}, respectively \emph{lower}, primitive segment as one whose vertices lie on or above, respectively on or below, the shifted diagonal.  We say that the path $P$ is upper if it is an initial segment itself, i.e. it has no primitive segments, or if it has only upper primitive segments, and we say that $P$ is lower otherwise.  Given a lower path $P$, we get an upper path $F(P)$, called the \emph{flip} of $P$, obtained from $P$ by keeping its initial and upper primitive segments fixed, and reflecting each of its lower primitive segments over the shifted diagonal line.  If $P$ is an upper path, define $F(P)=P$; its flip is itself.  Given a path system $\mathcal{P}\colon \mathcal{A}\rightarrow\mathcal{B}$ from some set of row vertices, say $\mathcal{A}$, on the diagonal, to some set of column vertices, say $\mathcal{B}$, lying on the shifted diagonal, we define its flipped path system $F(\mathcal{P})\colon\mathcal{A}\rightarrow\mathcal{B}$ to be the path system whose paths are the flips of paths in $\mathcal{P}$.  

In our case where $\mathcal{A}=\mathcal{A}^m_i$ and $\mathcal{B}=\mathcal{B}^m_i$, note that if $\mathcal{P}\colon\mathcal{A}\rightarrow\mathcal{B}$ is a sub-diagonal NE lattice path system, then so is its flipped path system $F(\mathcal{P})\colon\mathcal{A}\rightarrow\mathcal{B}$.  Indeed, in such a path system all lower primitive segments are contained in the square lying along the $x$-axis, with its top vertex $B_{\flo{i+2-m}_\star}\in\mathcal{B}_i^m$, and its diagonal lying along the shifted diagonal line $y=x+m-1$, and that square, closed under flips, lies completely under the diagonal line $y=x$.
  
Note, however, that flips need not preserve the vertex disjoint property; we say that a path system $\mathcal{P}$ is \emph{doubly vertex disjoint} if the path system $\mathcal{P}$ and its flipped path system $F(\mathcal{P})$ are both vertex disjoint.  Note that in our case, if a path system $\mathcal{P}\colon\mathcal{A}\rightarrow \mathcal{B}$ is doubly vertex disjoint then its permutation is necessarily the order reversing permutation, and hence its sign is constant.  
The following result is the first assertion of \Cref{introthm:WPM} from the Introduction.  It is certainly inspired by the ``traditional'' proof of \Cref{fact:LGV}.  
\begin{proposition}
	\label{lem:mDVD}
For any non-negative integers $i,m$ satisfying $0\leq i\leq \flo{3(m-1)}$, the determinant of the weighted path matrix 
$$W(G',\omega'\equiv 1, \mathcal{A}^m_i,\mathcal{B}^m_i)=\left(\frac{m}{3m-2-2p-2q}\binom{3m-2-2p-2q}{m-1-p-q}\right)_{\flo{i+2}_\star\leq p,q\leq \flo{i}}$$
is equal to $(-1)^{\flo{\flo{i+2}-\flo{i+2-m}_\star}}$ times the number of doubly vertex disjoint subdiagonal NE lattice path systems from $\mathcal{A}^m_i$ to $\mathcal{B}^m_i$. 
\end{proposition}
\begin{proof}
Let $\mathcal{N}=\mathcal{N}(\mathcal{A},\mathcal{B})$ denote the set of vertex disjoint path systems from $\mathcal{A}$ to $\mathcal{B}$ which are not doubly vertex disjoint.  Let us neglect for the moment the assertion of sign.  Then by \Cref{fact:LGV} and the remarks above, it suffices to show that there is a sign-reversing bijection $\phi\colon\mathcal{N}\rightarrow\mathcal{N}$, so that the path systems in $\mathcal{N}$ all cancel out with each other in the determinant, leaving only the doubly vertex disjoint path systems to be counted.  

To this end, fix a path system $\mathcal{P}\in\mathcal{N}$ and let $F(\mathcal{P})$ be its flip.  Let $c=(a,b)$ be the northern most, then eastern most vertex at which two (unique!) paths, say $F(P_r),F(P_s)\in F(\mathcal{P})$, intersect.  Let $F(S_r)\subset F(P_r)$ and $F(S_s)\subset F(P_s)$ be the primitive (or initial) segments containing $c$.  Since $\mathcal{P}$ is vertex disjoint, it follows that exactly one of $S_r$ or $S_s$ was lower, and the other is either initial or upper; after relabeling if necessary, we may assume that $S_r$ is the lower primitive segment.  For a point $z=(x,y)$ in the plane, let $\tilde{z}=(y-m+1,x+m-1)$ be its reflection across the shifted diagonal line $y=x-m+1$.  Setting 
\begin{align*}
	S_r= & v_j\rightarrow v_{j+1}\rightarrow\cdots v_d=\tilde{c}\rightarrow v_{d+1}\rightarrow \cdots v_J\\
	S_s= & w_k\rightarrow w_{k+1}\rightarrow \cdots w_e=c\rightarrow w_{e+1}\rightarrow \cdots w_K
\end{align*}
we have
\begin{align*}
		F(S_r)= & v_j=\tilde{v_j}\rightarrow \tilde{v_{j+1}}\rightarrow\cdots \tilde{v_d}=c\rightarrow \tilde{v_{d+1}}\rightarrow \cdots \tilde{v_J}=v_J\\
	F(S_s)= & w_k\rightarrow w_{k+1}\rightarrow \cdots w_e=c\rightarrow w_{e+1}\rightarrow \cdots w_K.
\end{align*}
Define new segments $S_r'$ and $S_s'$ by swapping the final legs after $c$ in the flipped segments $F(S_r)$ and $F(S_s)$, and then flipping back, i.e.
\begin{align*}
	S_r'= & v_j\rightarrow v_{j+1}\rightarrow\cdots v_d=\tilde{c}\rightarrow \tilde{w_{e+1}}\rightarrow \cdots \tilde{w_K}\\
	S_s'= & w_k\rightarrow w_{k+1}\rightarrow \cdots w_e=c\rightarrow \tilde{v_{d+1}}\rightarrow \cdots \tilde{v_J}.
\end{align*}
Then to define the new path system $\phi(\mathcal{P})\colon\mathcal{A}\rightarrow\mathcal{B}$, we simply replace segments $S_r\subset P_r$ and $S_s\subset P_s$ by the new segments $S_r'$ and $S_s'$ to obtain new paths $P_r'$ and $P_s'$, and keep all other paths the same.  See \Cref{fig1}.
\begin{figure}[h!]
	\includegraphics[width=12cm]{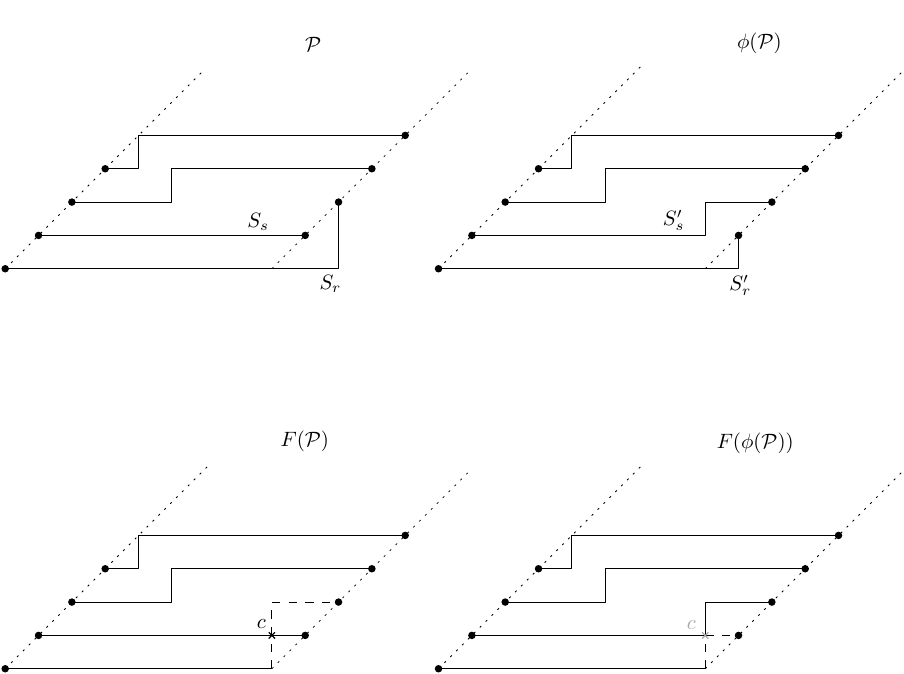}
	\caption{A path system, its image under the involution $\phi$, and their associated flipped path systems}
	\label{fig1}
\end{figure}    

Note that $\mathcal{P}$ and $\phi(\mathcal{P})$ have opposite signs since their corresponding permutations differ by a single transposition.  Also note that $\phi$ is invertible; in fact it is an involution.  It remains to see that $\phi(\mathcal{P})$ is vertex disjoint, or in other words that $\phi(\mathcal{P})\in \mathcal{N}$.  If not, then the new crossing must occur in one of the new segments $S_r'$ or $S_s'$, and then after $\tilde{c}$ or $c$, on the new part.  Suppose, for concreteness, that the crossing occurs at some vertex $\tilde{w_b}\in S_r'$ for some $k+1\leq b\leq K$.  Then $\tilde{w_b}$ lies in paths $P_r'$ and $P_t'$ in $\phi(\mathcal{P})$ for some $t\neq r$.  That means that in the original flipped path system $F(\mathcal{P})$, $w_b$ belongs to both paths $F(P_r)$ and $F(P_t)$.  On the other hand, since $w_b$ comes after $c$, $w_b$ is more northerly-easterly of $c$, in contradiction to our original choice of $c$.  Therefore there can be no crossing of paths in the path system $\phi(\mathcal{P})$, and hence $\phi(\mathcal{P})\in\mathcal{N}$.  Thus we have produced a sign-reversing bijection $\phi\colon\mathcal{N}\rightarrow\mathcal{N}$, as desired.

Finally, note that for any doubly vertex disjoint path system $\mathcal{P}$ from $\mathcal{A}^m_i$ to $\mathcal{B}^m_i$, either $\mathcal{P}$ or its flipped path system $F(\mathcal{P})$ must be a vertex disjoint path system completely contained in the parallelogram bounded by the diagonal lines $y=x$ and $y=x-m+1$ and horizontal lines $y=\flo{i+2-m}_\star$ and $y=m-1-\flo{i+2-m}_\star$.  Since $\mathcal{P}$ and $F(\mathcal{P})$ have the same sign, it follows that the permutation of $\mathcal{P}$ is the order-reversing permutation on the indexing set $\{1,\ldots,h_i\}$ where $h_i=\flo{i+2}-\flo{i+2-m}_\star=\#\mathcal{A}^m_i$, which has sign $(-1)^{\flo{h_i}}$
and the result follows.
\end{proof}

It follows from \Cref{lem:mDVD} that the determinant of the weighted path matrix $W(G',\omega'\equiv 1,\mathcal{A}^m_i,\mathcal{B}^m_i)$, and hence the determinant of the higher Hessian matrix $\Hess_i(F_m,\mathcal{E}_i)$ is nonzero precisely when the number of doubly vertex disjoint subdiagonal NE lattice path systems from $\mathcal{A}^m_i$ to $\mathcal{B}^m_i$ is nonzero; denote this number by $N(i,m)$.  Note that if $0\leq i\leq \flo{3(m-1)}$, then $$h_i=\flo{i+2}-\flo{i+2-m}_\star=H(m,2)_i=\#\mathcal{A}_i^m.$$ 
The following result amounts to the second assertion of \Cref{introthm:WPM}.
\begin{proposition}
	\label{cor:Lefschetz}
	For each $0\leq i\leq \flo{3(m-1)}$, the number $N(i,m)$ is nonzero if and only if $2\cdot h_i\leq m$. 
\end{proposition}
\begin{proof}
	Note that $N(i,m)\neq 0$ if and only if there exists a vertex disjoint NE lattice path contained completely inside the closed parallelogram bounded by the diagonal lines $y=x$ and $y=x-m+1$ and the horizontal lines $y=\flo{i+2-m}_\star$ and $y=m-1-\flo{i+2-m}_\star$.  Assuming first that $2\cdot h_i\leq m$, then we can construct a vertex disjoint NE lattice path system from each of the $h_i$ vertices in $\mathcal{A}^m_i$ to $\mathcal{B}^m_i$ using the $\flo{m}$ disjoint tubes formed by adjacent diagonal lines $y=x-2p$ and $y=x-2p-1$ for $0\leq p\leq \flo{m}$.  Conversely, if $2\cdot h_i>m$, then by counting the number of lattice points contained in the closed parallelogram, we see that there are not enough lattice points to contain $h_i$ vertex disjoint NE lattice paths from $\mathcal{A}^m_i$ to $\mathcal{B}^m_i$, and hence $N(i,m)$ must equal zero.
\end{proof}

Finally we are in a position to prove \Cref{introthm:HRPn2} from the Introduction.
\begin{proposition}
	\label{prop:ThmB}
	\begin{enumerate}
		\item If $m$ is even then $A(m,2)$ satisfies SLP.  If $m$ is odd then $A(m,n)$ satisfies SLP$_{m-1}$.
		\item The algebra $A(m,2)$ satisfies HLP for all $m$.
		\item If $m$ is even then $A(m,2)$ satisfies the complex HRR.  If $m$ is odd then $A(m,2)$ satisfies complex HRR$_{m-1}$.
	\end{enumerate}
\end{proposition}
\begin{proof}
	First assume that $m$ is even.  Then it follows from \Cref{prop:HilbertFunction} that $2\cdot h_i\leq m$ for all $0\leq i\leq \flo{3(m-1)}$.  It follows from \Cref{lem:PHessWPM}, \Cref{lem:mDVD}, and \Cref{cor:Lefschetz} that for each $0\leq i\leq \flo{3(m-1)}$ the $i^{th}$ Hessian matrix $\Hess_i(F_m,\mathcal{E}_i)$ is nonsingular, and hence from \Cref{lem:Hessian} and \Cref{lem:HRP} it follows that $A(m,2)$ must have SLP.
	
	If $m$ is odd, it follows from \Cref{prop:HilbertFunction} that $h_0\leq h_1\leq \cdots\leq h_{m-1}=\flo{m}$ and $h_{m-1}=\min\{h_{m-1},\ldots,h_{2(m-1)}\}$.  Therefore it again follows from \Cref{lem:PHessWPM}, \Cref{lem:mDVD}, and \Cref{cor:Lefschetz} that for each $0\leq i\leq \flo{3(m-1)}$ the $i^{th}$ Hessian matrix $\Hess_i(F_m,\mathcal{E}_i)$ is nonsingular, and hence $A(m,2)$ satisfies SL$_{m-1}$.  It follows from \Cref{rem:HLP} that $A(m,2)$ satisfies HLP for all $m$.
	
	That $A(m,2)$ satisfies complex HRR if $m$ is even, or complex HRR$_{m-1}$ if $m$ is odd follows directly from \Cref{lem:mDVD} and \Cref{lem:HRP}. 
\end{proof}
\begin{remark}
	\label{rem:NOHRR}
	As pointed out in \Cref{rem:preclude}, $A(m,2)$ satisfying the complex HRR$_r$ precludes it satisfying HRR$_r$ for any $r\geq 2$.  In particular it follows from \Cref{cor:Lefschetz} that the ring $A(m,2)$ cannot satisfy HRP for any $m$ satisfying $r=\flo{3(m-1)}\geq 2$, or equivalently for any $m\geq 3$.  On the other hand, for $m=2$, $A(2,n)$ does satisfy HRR (and complex HRR!), and in fact it is the cohomology ring of a smooth complex projective algebraic variety, namely the complex projective space $\C\P^3$.  In contrast, it follows from \Cref{prop:ThmB} that $A(m,2)$ cannot be the cohomology ring of any smooth complex projective algebraic variety for any $m>2$!    
\end{remark}

\begin{example}
	\label{ex:m7i56}
Take $m=5$.  Then 
$$A(5,2)=\frac{\Q[e_1,e_2]}{(f_5(e_1,e_2),e_2^5)}=\frac{\Q[e_1,e_2]}{\Ann(F_5(E_1,E_2))}$$
where 
$$f_5(e_1,e_2)=e_1^5-5e_1^3e_2+5e_1e_2^2$$
and 
$$F_5(E_1,E_2)=\frac{E_1^4E_2^4}{4!4!}+5\frac{E_1^6E_2^3}{6!3!}+20\frac{E_1^8E_2^2}{8!2!}+75\frac{E_1^{10}E_2}{10!}+275\frac{E_1^{12}}{12!}.$$
Take $i=3$.  Then the ordered monomial basis is 
$$\mathcal{E}_3=\left\{b^3_1=e_1^3, b^3_2=e_1e_2\right\}$$
and the $3^{rd}$ Hessian evaluated at $\mathbf{c}=(1,0)$ is 
$$\Hess_3(F_5,\mathcal{E}_3)|_{\mathbf{c}}=\left(\begin{array}{cc} e_1^6\circ F_5|_{\mathbf{c}} & e_1^4e_2\circ F_5|_{\mathbf{c}}\\ e_1^4e_2\circ F_5|_{\mathbf{c}} & e_1^2e_2^2\circ F_5|_{\mathbf{c}}\\ \end{array}\right)=\frac{1}{6!}\cdot \left(\begin{array}{cc}275 & 75\\
75 & 20\\\end{array}\right)\overset{\det}{\mapsto} -\frac{125}{6!}.$$
	\begin{figure}[h!!]
	\includegraphics[width=10cm]{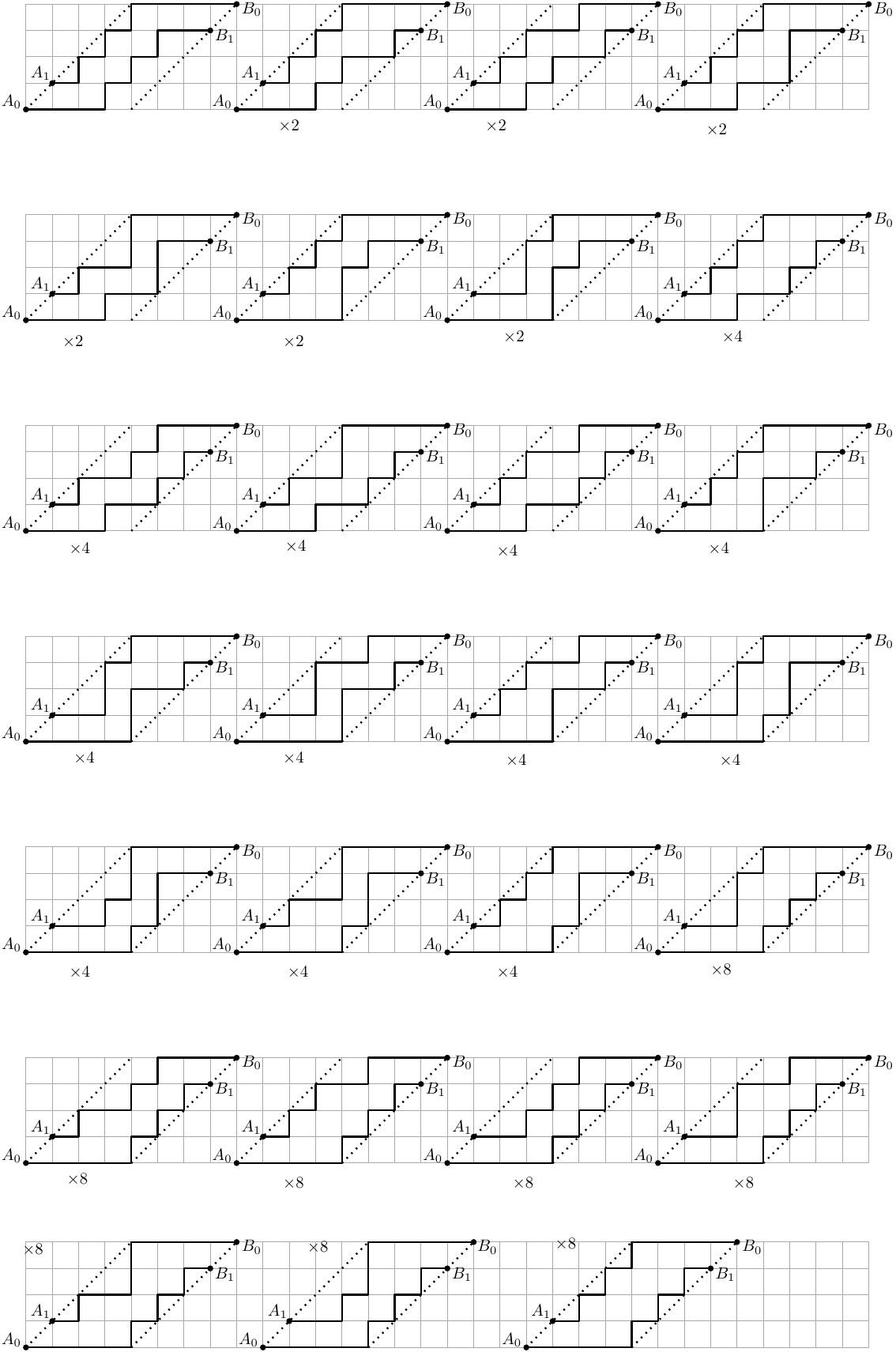}
	\caption{Doubly vertex disjoint path systems from $\mathcal{A}^5_3$ to $\mathcal{B}^5_3$, counted with multiplicities (125 total).}
	\label{fig:paths}
	\end{figure}
The $125$ in the determinant counts the number of doubly vertex disjoint subdiagonal NE lattice path systems from $\mathcal{A}^5_3=\{A_0=(0,0),A_1=(1,1)\}$ to $\mathcal{B}^5_3=\{B_0=(8,4),B_1=(7,3)\}$ shown in \Cref{fig:paths} (actually shown are those upper path systems that lie completely inside the parallelogram; every lower path system may be transformed into a unique upper path system by flipping the lower primitive segments, hence the multiplicities attached--a factor of $2$ for every primitive segment).

On the other hand, if we take $i=4$, then there are no doubly vertex disjoint subdiagonal NE lattice path systems from 
$\mathcal{A}^5_4=\{A_p=(p,p) \ | \ 0\leq p\leq 2\}$ and $\mathcal{B}^5_4=\{B_q=(8-q,4-q) \ | \ 0\leq q\leq 2\}$ since $h_4=3$ and $2\cdot h_4>m=5$, and hence we deduce that the determinant of the $4^{th}$ Hessian evaluated at $\mathbf{c}=(1,0)$ must be zero:
$$\det\left(\Hess_4(F_5,\mathcal{E}_4)|_{\mathbf{c}}\right)=0.$$
Shown in \Cref{fig:cancel} are two vertex disjoint path systems that cancel out in the determinant, as in \Cref{lem:mDVD}.
\begin{figure}[h!]
	\includegraphics[width=10cm]{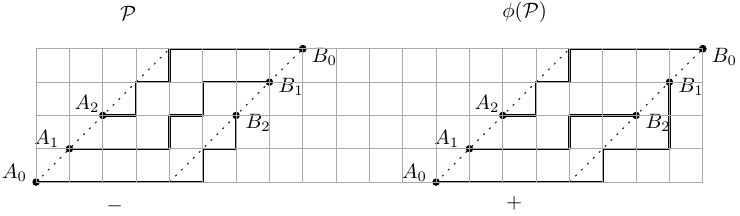}
	\caption{Two path systems that cancel out in the determinant $\det\left(\Hess_4\left(F_5\right)\right)$.}
	\label{fig:cancel}
\end{figure}
Of course we should expect this by looking at the Hilbert function
$$H(A(5,2))=(1,1,2,2,3,2,3,2,3,2,2,1,1);$$
indeed, since $\dim_{\Q}(A_4)=3$ and $\dim_{\Q}(A_5)=2$, it follows that the multiplication map $\times a^{4}\colon A_4\rightarrow A_8$ must have nonzero kernel. 
\end{example}



\bibliographystyle{alpha}

\appendix
\section{Connections to Invariant Theory and Partitions}
\label{sec:Appendix}
\subsection{Invariant Theory}
	Let $\F=\Q(\zeta)$ denote the field of rational numbers adjoined with $\zeta$ a primitive $m^{th}$ root of unity.  Let $G(m,1,n)$ be the group of $n\times n$ matrices over $\F$ consisting of permutation matrices whose nonzero entries are $m^{th}$ roots of unity.  This is equal to the semi-direct product  
$$G(m,1,n)=\left\{\left.\left(\begin{array}{ccc} \lambda_1 & \cdots & 0\\ \vdots & \ddots & \vdots\\ 0 & \cdots & \lambda_n\\ \end{array}\right) \right| \lambda_i^m=1\right\}\rtimes \mathfrak{S}_n$$
where $\mathfrak{S}_n$ is the subgroup of permutation matrices whose nonzero entries are $1$'s.  The group $G(m,1,n)$ is a pseudo-reflection group generated by permutation matrices $\sigma\in\mathfrak{S}_n$ together with the diagonal matrix $t=\left(\begin{array}{ccc} \zeta & \cdots & 0\\ \vdots & \ddots & \vdots\\ 0 & \cdots & 1\\ \end{array}\right)$.  Then $G(m,1,n)$ acts on the polynomial ring $S=\F[x_1,\ldots,x_n]$ in the usual way with $\sigma(x_i)=x_{\sigma(i)}$ and $t(x_i)=\begin{cases} \zeta\cdot x_i & \text{if} \ i= 1\\
	x_i & \text{if} \ i\neq 1\\ \end{cases}$.

The elementary symmetric functions $e_i=e_i(x_1,\ldots,x_n)=\sum_{1\leq k_1<\cdots<k_i\leq n}x_{i_1}\cdots x_{i_n}\in R_\F$ are a fundamental set of invariants of the symmetric group $\mathfrak{S}_n$ acting on $S$ (or even $R=\Q[x_1,\ldots,x_n]$) as above, meaning that the subalgebra they generate $S^{\mathfrak{S}_n}=\F[e_1,\ldots,e_n]\subset S$ is a polynomial algebra that contains every other symmetric function, e.g. \cite[Theorems 1.1.1 and 1.1.2]{Smith}.  
Define the $i^{th}$ $m^{th}$ power elementary symmetric function $e_i(m)$ to be the $i^{th}$ elementary symmetric function evaluated at the $m^{th}$ powers of the variables, i.e.
$e_i(m)=e_i(x_1^m,\ldots,x_n^m)$.  Note that $e_i(m)$ is invariant under the larger group action of $G(m,1,n)$ for all $1\leq i\leq n$.

Since $e_1,\ldots,e_n$ are algebraically independent over $\F$, and $x_1^m,\ldots,x_n^m$ are algebraically independent and $\F$, it follows that $e_1(m),\ldots,e_n(m)$ are also algebraically independent over $\F$.  Since $e_1(m)\otimes 1,\ldots,e_n(m)\otimes 1\in S$ are algebraically independent over $\F$ and $\prod_{i=1}^n\deg(e_i(m))=|G(m,1,n)|=m^n\cdot n!$, it follows from \cite[Proposition 7.4.2]{Smith} that the invariant polynomials $e_1(m),\ldots,e_n(m)$ are a fundamental set of invariants for the pseudo reflection group $G(m,1,n)$, meaning that the subalgebra generated by them is a polynomial algebra that contains all other $G(m,1,n)$-invariant polynomials; in symbols we write 
$$\left(S\right)^{G(m,1,n)} =\F[e_1(m),\ldots,e_n(m)].$$  
It follows from a theorem of E. Noether, e.g. \cite[Theorem 2.3.1]{Smith} that the inclusion $S^{G(m,1,n)}\subset S$ is a finite extension of rings meaning that $S$ is finitely generated as a module over $S^{G(m,1,n)}$, and likewise for $S^{\mathfrak{S}_n}\subset S$.  It follows that the inclusion $S^{G(m,1,n)}\subseteq S^{\mathfrak{S}_n}$ is also a finite extension.  Indeed, if $f_1,\ldots,f_N\in S$ are a finite generating set for $S$ as a module over $S^{G(m,1,n)}$ then their symmetrizations under $\mathfrak{S}_n$, $f_1^\sharp,\ldots,f_n^\sharp$ where $f_i^\sharp = \frac{1}{n!}\sum_{\sigma\in\mathfrak{S}_n}\sigma(f_i)$, must be a finite generating set for $S^{\mathfrak{S}_n}$ as a module over $S^{G(m,1,n)}$.  This implies that the quotient algebra 
$$\frac{S^{\mathfrak{S}_n}}{S^{G(m,1,n)}_+\cdot S^{\mathfrak{S}_n}}=\frac{\F[e_1,\ldots,e_n]}{(e_1(m),\ldots,e_n(m))}=A(m,n)\otimes_\Q\F$$
is finite dimensional over $\F$, and hence is an Artinian complete intersection over $\F$.  Finally, since $\F\supseteq \Q$ is a finite field extension, it follows that $A(m,n)$ is finite dimensional over $\Q$, and hence must also be an Artinian complete intersection over $\Q$.

One can further show that if we define the coinvariant algebra of $G(m,1,n)$ to be the quotient algebra 
$$S_{G(m,1,n)}=\frac{S}{S^{G(m,1,n)}_+\cdot S}$$
then, and endow it with the induced action by the subgroup $\mathfrak{S}_n$, then the $\mathfrak{S}_n$-invariant subring of the $G(m,1,n)$-coinvariant ring satisfies
$$\left(S_{G(m,1,n)}\right)^{\mathfrak{S}_n}=\frac{S^{\mathfrak{S}_n}}{S^{G(m,1,n)}_+\cdot S^{\mathfrak{S}_n}}=A(m,n)\otimes_\Q\F.$$
In particular, we recover a special case of a theorem of S. Goto \cite{Goto}:  if $A$ is a complete intersection algebra and $G$ is a finite pseudo-reflection group acting linearly on $A$, then the invariant ring $A^G$ must also be a complete intersection.

For $m=2$, the group $G(2,1,n)$ is a real reflection group, and is in fact the Weyl group of type $B_n$.  It follows from results of Borel and Bernstein-Gelfand-Gelfand \cite[Proposition 1.3 and Theorem 5.5]{BGG} that 
$$A(2,n)\cong H^{2\bullet}(G/P,\Q)$$
where $G/P$ is a smooth complex projective algebraic variety (a partial flag variety of type $B_n$).  It follows from Hodge theory on K\"ahler manifolds, e.g. \cite{Huybrechts}, that $A(2,n)$ satisfies SLP and HRR for all $n$.  This is essentially Stanley's proof of Almkvist's conjecture for $m=2$ \cite[Theorem 3.1]{Stanley}.           
\subsection{Partitions}
Recall that an integer partition $\lambda$ is a weakly decreasing sequence of non-negative integers $\lambda=(\lambda_1,\ldots,\lambda_e)$; the individual integers $\lambda_k$ are called the parts of $\lambda$, the length of the partition $\ell(\lambda)$ is the largest index $j$ for which $\lambda_j\neq 0$, the size of the partition $|\lambda|$ is equal to the sum of its parts $|\lambda|=\lambda_1+\cdots+\lambda_e$, and the empty partition is the unique partition of size $0$.  Let $\mathsf{P}$ denote the set of all integer partitions and let $\mathsf{P}_k$ denote the subset of partitions of size $k$.

Consider the two parameter family $\mathsf{P}(m,n)\subset\mathsf{P}$ consisting of integer partitions with part sizes at most $n$ and the number of repeated part sizes at most $m-1$, i.e.
$$\mathsf{P}(m,n)=\left\{\lambda\in\mathsf{P} \ | \ \lambda_k\leq n, \ \lambda_{k+m-1}<\lambda_k, \ \forall k\right\}.$$  
For example, $\mathsf{P}(3,2)$ is the set of partitions with part sizes at most $2$ and number of repeated parts at most $3-1=2$, that is:
$$\mathsf{P}(3,2)=\left\{\emptyset, (1), (2), (1,1), (2,1), (2,2), (2,1,1), (2,2,1), (2,2,1,1)\right\}.$$

The generating function of $\mathsf{P}(m,n)$ is the polynomial 
$P(m,n;t)=\sum_{k\geq 0} p_k\cdot t^k$ where $p_k=\#\mathsf{P}_k$.  Note that every partition $\lambda\in \mathsf{P}(m,n)_k$ is uniquely determined by a sequence of  integers $(j_1,\ldots,j_n)$ where $0\leq j_i\leq m-1$ and $\lambda$ has $j_i$ parts of size $i$ for $1\leq i\leq n$, and $j_1+2j_2+\cdots+nj_n=k$.  On the other hand, every such sequence of integers contributes exactly one $t^k$ term in the expansion of the polynomial
$$H(m,n;t)=\prod_{i=1}^n\left(1+t^i+t^{2i}+\cdots+t^{(m-1)i}\right).$$
It follows that the Hilbert function of the algebra $A(m,n)$ is equal to the generating function of the set of partitions $\mathsf{P}(m,n)$, i.e. 
$H(m,n;t)=P(m,n;t)$ for all $m,n$.

In fact, results of B. Totaro \cite{Totaro} show that this connection runs much deeper.  He has shown that the algebra $A(m,n)$ has a basis of (eqivalence classes of) Hall-Littlewood symmetric polynomials, specialized at the $m^{th}$ root of unity and indexed by the partitions $\mathsf{P}(m,n)$, i.e.
$$\left\{\left[Q_\lambda(x_1,\ldots,x_n, \zeta)\right]\in A(m,n) \ | \ \lambda\in \mathsf{P}(m,n)\right\}.$$
It remains an intriguing and challenging open problem to describe their structure constants and to find their associated ``complex Schubert calculus''.  As a first step in this direction, Totaro computes the following ``degree formula'' for the algebra $A(m,n)$ \cite[Theorem 0.1]{Totaro}:
$$e_1^{\binom{n+1}{2}(m-1)}\equiv m^{\binom{n}{2}}\frac{\left(\binom{n+1}{2}(m-1)\right)! 1!2!\cdots(n-1)!}{(m-1)!(2m-1)!\cdots(nm-1)!}\cdot \left(e_1\cdots e_n\right)^{m-1}.$$ 
Of course, for $n=2$, this constant in front of the socle generator $(e_1\cdots e_n)^{m-1}$ agrees with our formula for the $0^{th}$ Hessian of $F_m$ from \Cref{lem:Hessian1} (scaled by $d!$):
$$(3m-3)!\cdot \Hess_0(F_m,\mathcal{B}_0)=\left(\frac{m}{3m-2}\binom{3m-2}{m-1}\right).$$

\end{document}